\documentclass{article}
\usepackage{graphicx} 
\usepackage{multirow}%
\usepackage{amsmath,amssymb,amsfonts}%
\usepackage{amsthm}%
\usepackage{mathrsfs}%
\usepackage[title]{appendix}%
\usepackage{xcolor}%
\usepackage{textcomp}%
\usepackage{manyfoot}%
\usepackage{booktabs}%
\usepackage{algorithm}%
\usepackage{algorithmicx}%
\usepackage{algpseudocode}%
\usepackage{listings}%

\theoremstyle{thmstyleone}
\newtheorem{Theorem}{Theorem}[section]
\newtheorem{Lemma}[Theorem]{Lemma}%
\newtheorem{Corollary}[Theorem]{Corollary}%

\theoremstyle{thmstyletwo}%
\newtheorem{Remark}[Theorem]{Remark}%

\theoremstyle{thmstylethree}%
\newtheorem{Definition}[Theorem]{Definition}%

\definecolor{darkred}{rgb}{0.8,0,0}
 
\title{Gelfand-Shilov spaces for extended Gevrey regularity}

\author{Nenad Teofanov\footnote{Faculty of Sciences, University of Novi Sad, Serbia. nenad.teofanov@dmi.uns.ac.rs}, Filip Tomi\'c\footnote{Faculty of Technical Sciences, University of Novi Sad, Serbia. filip.tomic@uns.ac.rs}, Milica \v Zigi\'c\footnote{Faculty of Sciences, University of Novi Sad, Serbia. milica.zigic@dmi.uns.ac.rs}}

\date{\today}

\begin{document}

\maketitle

\begin{abstract}
 We consider spaces of smooth functions obtained by relaxing Gevrey-type regularity and decay conditions. It is shown that these classes fit well within the general framework of the weighted matrices approach to ultradifferentiable functions. We examine equivalent ways of introducing Gelfand-Shilov spaces related to the extended Gevrey regularity and derive their nuclearity. In addition to the Fourier transform invariance property, we present their corresponding symmetric characterizations. Finally, we consider some time-frequency representations of the introduced classes of ultradifferentiable functions.  
\end{abstract}

{\em AMS subj class:} 46E10,  26E10, 46A11 

{\em Keywords and phrases:} Gevrey sequences, ultradifferentiable functions, Gel\-fand-Shilov spaces, nuclearity, time-frequency representations

\section{Introduction} \label{Sec1}

Traditionally, Gelfand-Shilov spaces 
of ultradifferentiable functions 
are introduced as sets of all $f\in C^\infty (\mathbb R^n)$ such that for a given $t,s >0$,
\begin{equation}\label{Eq:GeShSpDef}
\sup _{x\in \mathbb R^n} |x^\alpha \partial^\beta f(x)| 
\leq C h^{|\alpha  + \beta |}\alpha !^t \beta !^s, \qquad \alpha, \beta \in \mathbb{N}_0 ^n,
\end{equation} 
for some constants $C,h>0$, see \cite{GelfandShilov}.  In other words,
the regularity and decay properties of $f$
are controlled by Gevrey sequences 
$(\beta !^s)_{\beta \in \mathbb{N}_0 ^n},$ and $(\alpha !^t)_{\alpha \in \mathbb{N}_0 ^n}$ respectively.

These spaces are widely used in regularity theory for partial differential equations, e.g., to describe exponential decay and the holomorphic extension of solutions to globally elliptic equations \cite{CGR-1, CGR-2}, and in the regularizing properties of the Boltzmann equation \cite{LMPSX}. For additional examples, we refer to \cite{Gramchev, Nicola-Rodino, Cordero-Rodino}.

More general Gelfand-Shilov type spaces  are obtained by using various types of defining sequences
apart from the Gevrey sequences.

In this paper, we introduce Gelfand-Shilov type spaces through regularity and decay conditions that
go beyond the Gevrey sequences, by considering two-parameter dependent sequences of the form
$(p^{\tau p^{\sigma} })_{p \in \mathbb{N}_0 }$, $\tau > 0$, $\sigma > 1$. 
Such sequences have  been recently used to describe a convenient extension of Gevrey spaces, 
\cite{Garrido-2,PTT-01,PTT-03}. 
For an overview and some applications of the extended Gevrey classes, we refer to \cite{TTZ2024}. 
We note that the extended Gevrey classes are defined by regularity conditions, without reference to decay, whereas the elements of the spaces considered here, in addition to regularity, also enjoy suitable decay properties.

A commonly used approach to  ultradifferentiable functions, based on  defining sequences,
was introduced in \cite{Komatsuultra1}. Another widely used approach is the one based on the so-called weight functions,
as considered in  \cite{BMT}. Recently, Rainer and Schindl proposed an approach that unifies both classical methods, \cite{RS}. It is based on the concept of weight matrices, or multi-indexed weight sequences in the terminology of \cite{DNV2021}. A general framework that provides a unified treatment of Gelfand-Shilov spaces defined via weight sequences or weight functions was very recently presented in  \cite{BJOS,DNV2024}. There, one can find general results on the inclusion relation between the considered global ultradifferentiable classes.

Our aim is to study a new type of Gelfand–Shilov spaces that combines the general approach based on weight matrices with the framework of extended Gevrey regularity. This is achieved by considering the so-called limit classes, i.e., suitable unions and intersections with respect to the parameter, see Definition \ref{Def:extGS}.

In this way, we obtain prominent examples of global ultradifferentiable classes defined by general weighted matrix approach, which go beyond the classical theory.
At the same time, by considering specific sequences, we are able to derive results that cannot be extracted from the general theory,
while also extending those  that are well-known in the context  of the Gevrey type regularity.

To be more precise, we perform a thorough examination of general conditions from \cite{DNV2021} that provide nuclearity property of the considered spaces. Thereafter, we show the Fourier transform invariance, and symmetric characterization of Gelfand-Shilov spaces for extended Gevrey regularity \'a la Chung-Chung-Kim \cite{Chung 1996}. We emphasize that the sequence
$(p^{\tau p^{\sigma} })_{p \in \mathbb{N}_0 }$, $\tau > 0$, $\sigma > 1$, does not satisfy Komatsu's $(M.2)$ condition:
\begin{equation*} \label{ultradiff-stability}
(M.2) \qquad (\exists A,B > 0) \quad M_{p+q}\leq AB^{p+q} M_p M_q,  \qquad p,q\in \mathbb N_0.
\end{equation*}
This invokes nontrivial modifications to the existing proofs for classical Gelfand-Shilov type spaces.

Furthermore, we study the properties of several time-frequency representations acting on Gelfand–Shilov spaces via extended Gevrey sequences. This leads to characterizations that extend those given in, e.g., \cite{GZ, Teo2018, Toft-2012}, in the context of classical Gelfand–Shilov spaces.

\section{Preliminaries}  \label{Sec2}

\subsection{Weight  sequences and matrices}

We first recall the notion of a weight matrix \cite{RS}, or multi-indexed weight sequence system in terms of \cite{DNV2021, DNV2024}.

\begin{Definition}
Let $M = ( M_{\alpha})_{\alpha\in \mathbb N_0 ^n}$ 
be a weight sequence, i.e. $M_{\alpha} >0$, $\alpha\in \mathbb N_0 ^n$,
$M_0 = 1$, and 
$ \lim_{|\alpha| \rightarrow \infty} (M_{\alpha})^{1/ |\alpha|} = \infty$.

A family 
$\mathcal{M} =\{ M ^{(\lambda)} \colon \lambda>0\}$ of weight sequences is a weight matrix if
$ M_{\alpha} ^{(\lambda)} \leq  M_{\alpha} ^{(\mu)} $ for all $\alpha\in \mathbb N_0 ^n$
and $\lambda \leq \mu$.    
\end{Definition} 

We say that the sequence $M = ( M_{\alpha})_{\alpha\in \mathbb N_0 ^n}$ 
is an isotropic sequence if 
$  M_{|\alpha|} = M_{p}, $  $p\in\mathbb N_0$, for some sequence
$ ( M_{p})_{p\in\mathbb N_0} $.

\par 

Next we introduce isotropic sequences related to extended Gevrey regularity.


Let $\tau>0$ and $\sigma>1.$ By  $M^{(\tau)} = 
(M_{p}^{\tau,\sigma})_{p\in\mathbb N_0}$ we denote 
 the sequence of positive numbers given by
\begin{equation} \label{eq:the-sequence}
    M_{p}^{\tau,\sigma}=p^{\tau p^\sigma},\quad p\in \mathbb{N},
     \qquad M_0^{\tau,\sigma}=1.
\end{equation}

\par

We are interested in  the weight matrix $\mathcal{M}_\sigma,$ $\sigma>1$, given by
    \begin{equation}\label{eq:the-weight-matrix}
\mathcal{M}_\sigma=\left\{M^{(\tau)} =  (M_{p}^{\tau,\sigma})_{p\in\mathbb N_0}\colon \tau>0\right\}.
    \end{equation}

\par

The main properties of $(M_{p}^{\tau,\sigma})_{p\in\mathbb N_0}$,
$\tau>0$, $\sigma>1$, related to Komatsu's theory of ultradistributions
are given in the next lemma (cf. \cite[Lemmas 2.2 and 3.1]{PTT-01}, 
\cite[Lemma 2]{TTZ2024}).

\begin{Lemma} \label{osobineM_p_s}
Let $\tau>0$, $\sigma>1$, $M_0 ^{\tau,\sigma}=1$,
and $M_p ^{\tau,\sigma}=p^{\tau p^{\sigma}}$, $p\in \mathbb N$.
Then the following properties hold:
\begin{itemize}
\item[$(M.1)$] \hspace{1em} 
$ (M_p^{\tau,\sigma})^2\leq M_{p-1}^{\tau,\sigma}M_{p+1}^{\tau,\sigma}, \;$ $p\in \mathbb N, $ \medskip
\item[$\widetilde{(M.2)}$]  \hspace{1em}   
$ M_{p+q}^{\tau,\sigma}\leq C^{p^{\sigma} + q^{\sigma}}
M_p^{\tau 2^{\sigma-1},\sigma}M_q^{\tau 2^{\sigma-1},\sigma},\;$
$ p,q\in \mathbb N_0,\quad$
for some constant $ C\geq 1$,\medskip
\item[$\widetilde{(M.2)'}$] \hspace{1em}
$ M_{p+1}^{\tau,\sigma}\leq C ^{p^{\sigma}} M_p^{\tau,\sigma}, \;$ 
$ p\in \mathbb N_0, \quad$
for some constant $C\geq 1$, \medskip
\item[$(M.3)'$] \hspace{1em}
$ \displaystyle   \sum\limits_{p=1}^{\infty}\frac{M_{p-1}^{\tau,\sigma}}{M_p^{\tau,\sigma}} <\infty.
$
\end{itemize} 
\end{Lemma}

In addition,  $
\left ( \left (\frac{M_{p}^{\tau,\sigma}}{p!} \right )^{1/p} \right )_{p\in\mathbb N_0}$
is an almost increasing sequence, i.e.
$$
\left(\frac{M_{p}^{\tau,\sigma}}{p!} \right)^{1/p} 
\leq 
C \left(\frac{M_{q}^{\tau,\sigma}}{q!} \right)^{1/q}, \qquad
p\leq q,
$$
which implies 
$ \lim_{p \rightarrow \infty} (M_{p})^{1/ p} = \infty$
(cf. \cite{TTZ2024}).

We note that (M.1) implies 
\begin{equation} \label{M.1'}
M_p^{\tau,\sigma}M_q^{\tau,\sigma}\leq M_{p+q}^{\tau,\sigma},\qquad p,q\in\mathbb N_0,
\end{equation}
and that $(M.3)'$ follows from the estimate
\begin{align}\label{M.3''}
        \frac{M_{p-1}^{\tau,\sigma}}{M_p^{\tau,\sigma}}\leq\frac{1}{(2p)^{\tau(p-1)^{\sigma-1}}},\quad p\in\mathbb N.
\end{align}
The proof of \eqref{M.3''} can be found in e.g. \cite[Appendix A]{TTZ2024}.

More properties of  $(M_{p}^{\tau,\sigma})_{p\in\mathbb N_0}$ in the context of weight matrices approach to ultradifferentiable functions (in the sense of A. Rainer and G. Schindl)
can be found in e.g. \cite{Garrido-2}.  

\par 

The sequence $(M_{p}^{\tau,\sigma})_{p\in\mathbb N_0}$ gives rise to 
extended Gevrey classes $\mathcal E_{\tau,\sigma}(\mathbb R ^n)$, also called 
Pilipovi\'c-Teofanov-Tomi\'c classes (PTT-classes) in \cite{Garrido-2,
Garrido-Lastra-Sanz}. 
We refer to \cite{TTZ2024} for a recent survey on PTT-classes and their relation to
classes of ultradifferentiable functions given by the weight matrix approach.

In fact, we use the weight matrix approach and consider PTT-limit classes.
As it is noted in \cite{Garrido-2} the main reason to consider 
these classes is their  stability under the action of ultradifferential operators.

\subsection{Weight functions and the Lambert function}

Recall, a non-negative, continuous, even, and increasing on $[0,\infty)$ function $\omega$ defined on $\mathbb R$, $\omega(0)=0$, is called \emph{BMT (Braun-Meise-Taylor) weight function}, or simply \emph{weight function}, if it satisfies the following conditions:
\begin{itemize}
\item[($\alpha$)] $\omega(2t)=O(\omega(t)),\quad t\to \infty,$
\vspace{0.1cm}
\item[($\beta$)] $ \omega(t)=O(t),\quad t\to\infty,$
\vspace{0.1cm}
\item[($\gamma$)] $o(\omega(t))=\log t,\quad t \to \infty,$
\vspace{0.1cm}
\item [($\delta$)] $\varphi(t)=w(e^t),\quad$ is convex.
\end{itemize} 

Some classical examples of BMT functions are
\begin{equation}
\label{BMTexamples}
\omega (t) = \ln^{s}_+ |t|,\quad \omega (t)=\frac{|t|}{\ln^{s-1} (e+|t|)},\quad s>1,\; t\in \mathbb R^n,
\end{equation} where $\ln_+ k=\max\{0,\ln k\}$. Also, $\omega (t)=|t|^s$ is a weight function if and only if $0<s\leq 1$.

We say that functions $f$ and $g$ are equivalent if $f=O(g)$ and $g=O(f)$.

Next we introduce the Lambert $W$ function which 
describes the precise asymptotic behavior of associated function to the sequence $(M_p^{\tau,\sigma})_{p \in \mathbb N_0} $.

The \emph{Lambert $W$ function} is defined as the inverse of $z e^{z}$, $z\in {\mathbb C}$. By $W(x)$, we denote the restriction of its principal branch to $[0,\infty)$. It is used as a convenient tool to describe asymptotic behavior in different contexts. We refer to \cite{LambF,Mezo} for a review of some applications of the Lambert $W$ function in pure and applied mathematics.

Some basic properties of the Lambert function $W$ are given below:
\newline
$(W1)$ \hspace{1em} $W(0)=0$, $W(e)=1$, $W(x)$ is continuous, increasing and concave on $[0,\infty)$,
\newline
$(W2)$ \hspace{1em} $W(x e^{x})=x$ and $ x=W(x)e^{W(x)}$,  $x\geq 0$,
\newline
$(W3)$ \hspace{1em} $W$ can be represented in the form of the absolutely convergent series
\begin{equation*}
W(x)=\ln x-\ln (\ln x)+\sum_{k=0}^{\infty}\sum_{m=1}^{\infty}c_{km}\frac{(\ln(\ln x))^m}{(\ln x)^{k+m}},\quad x\geq x_0>e,
\end{equation*}
with suitable constants $c_{km}$ and  $x_0 $, wherefrom  the following  estimates hold:
\begin{equation}
\label{sharpestimateLambert}
\ln x -\ln(\ln x)\leq W(x)\leq \ln x-\frac{1}{2}\ln (\ln x), \quad x\geq e.
\end{equation}
The equality in \eqref{sharpestimateLambert} holds if and only if $x=e$.

Note that $(W2)$ implies
\begin{equation*}
\label{PosledicaLambert1}
W(x\ln x)=\ln x,\quad x>1.
\end{equation*}
By using $(W3)$ we obtain
\begin{equation*}
\label{PosledicaLambert1.5}
W(x)\sim \ln x, \quad x\to \infty,
\end{equation*}
and therefore
\begin{equation*}
\label{PosledicaLambert2}
W(C x)\sim W(x),\quad x\to \infty,
\end{equation*}
for any $C>0$.

\begin{Remark}
\label{RemarkBMT}
By \cite[Theorem 1]{TT0} it follows that 
$\displaystyle 
\omega(t)=\frac{\ln^{\sigma}(1+|t|)}{W^{\sigma-1}(\ln(1+|t|))}
$, $\sigma>1$, $t\in \mathbb R^n$, 
is equivalent to a BMT function, and therefore it satisfies $(\alpha)-(\gamma)$. 
Actually, in \cite{TT0}, the function $\ln_+|\cdot|$ is considered instead of  $\ln(1+|\cdot|)$, which does not change the conclusion.
\end{Remark} 



\section{Gelfand-Shilov type classes}

\subsection{Definition and its relatives}

Inspired by the original source \cite{GelfandShilov}, we consider separate and joint conditions on decay and regularity, introducing
isotropic Gelfand-Shilov type spaces
in terms of the weight matrix $\mathcal{M}_\sigma$, $ \sigma >1$,
as follows.

\begin{Definition}\label{Def:extGS}
Let $\sigma>1$, the weight matrix $\mathcal{M}_\sigma$ be given by 
\eqref{eq:the-weight-matrix}, and let $\varphi\in C^\infty(\mathbb R^n)$. Then
the Gelfand-Shilov spaces of Roumieu type related to $\mathcal{M}_\sigma$ are given by
\begin{align*}
\varphi\in \mathcal S^{\{\mathcal{M}_\sigma\}} (\mathbb R^n) \quad \Leftrightarrow \quad & (\exists \tau>0) \;(\forall \alpha\in\mathbb N_0^n)\;(\exists C_\alpha>0)\;(\forall \beta\in\mathbb N_0^n) \\
   &\sup_{x}|x^\alpha \partial^\beta \varphi(x)|\leq C_\alpha M^{\tau,\sigma}_{|\beta|};\\
    \varphi\in \mathcal S_{\{\mathcal{M}_\sigma\}} (\mathbb R^n) \quad \Leftrightarrow \quad & (\exists \tau>0) \;(\forall \beta\in\mathbb N_0^n)\;(\exists C_\beta>0)\;(\forall \alpha\in\mathbb N_0^n) \\
   &\sup_{x}|x^\alpha \partial^\beta \varphi(x)|\leq C_\beta M^{\tau,\sigma}_{|\alpha|};\\
\varphi\in \mathcal  S_{\{\mathcal{M}_\sigma\}}^{\{\mathcal{M}_\sigma\}} (\mathbb R^n) \quad \Leftrightarrow \quad & (\exists \tau>0)\;(\exists C>0)\;(\forall \alpha,\beta\in\mathbb N_0^n) \\
   &\sup_{x}|x^\alpha \partial^\beta \varphi(x)|\leq C M^{\tau,\sigma}_{|\alpha|} 
   M^{\tau,\sigma}_{|\beta|},
\end{align*}
and the Gelfand-Shilov spaces of Beurling type related to $\mathcal{M}_\sigma$ are given by
\begin{align*}
\varphi\in \mathcal S^{(\mathcal{M}_\sigma)} (\mathbb R^n) \quad \Leftrightarrow \quad & (\forall \tau>0) \;(\forall \alpha\in\mathbb N_0^n)\;(\exists C_{\alpha}>0)\;(\forall \beta\in\mathbb N_0^n) \\
   &\sup_{x}|x^\alpha \partial^\beta \varphi(x)|\leq C_{\alpha} M^{\tau,\sigma}_{|\beta|};\\
    \varphi\in \mathcal S_{(\mathcal{M}_\sigma)} (\mathbb R^n) \quad \Leftrightarrow \quad & (\forall \tau>0)\;(\forall \beta\in\mathbb N_0^n)\;(\exists C_{\beta}>0)\;(\forall \alpha\in\mathbb N_0^n) \\
   &\sup_{x}|x^\alpha \partial^\beta \varphi(x)|\leq C_{\beta}
   M^{\tau,\sigma}_{|\alpha|};\\
\varphi\in \mathcal S_{(\mathcal{M}_\sigma)}^{(\mathcal{M}_\sigma)} (\mathbb R^n) \quad \Leftrightarrow \quad & (\forall \tau>0) (\exists C>0)(\forall \alpha,\beta\in\mathbb N_0^n) \\
   &\sup_{x}|x^\alpha \partial^\beta \varphi(x)|\leq C
   M^{\tau,\sigma}_{|\alpha|}    M^{\tau,\sigma}_{|\beta|}.
\end{align*}

\end{Definition}

We will use $[\mathcal{M}_\sigma]$ as a common notation for 
$\{ \mathcal{M}_\sigma \}$ or $(\mathcal{M}_\sigma )$, i.e.
$\mathcal S_{[\mathcal{M}_\sigma]}^{[\mathcal{M}_\sigma]} $ is either
$ \mathcal  S_{\{\mathcal{M}_\sigma\}}^{\{\mathcal{M}_\sigma\}}$
or $ \mathcal S_{(\mathcal{M}_\sigma)}^{(\mathcal{M}_\sigma)}$, and so on.

\par

By Definition \ref{Def:extGS} we immediately see that 
\begin{equation}\label{eq:incl}
    \mathcal S_{[\mathcal{M}_\sigma]}^{[\mathcal{M}_\sigma]} (\mathbb R^n) \subseteq \mathcal S_{[\mathcal{M}_\sigma]} (\mathbb R^n)\cap \mathcal S^{[\mathcal{M}_\sigma]} (\mathbb R^n),
\end{equation}
and $\mathcal S_{(\mathcal{M}_\sigma)}^{(\mathcal{M}_\sigma)} (\mathbb R^n) \subseteq \mathcal S_{\{\mathcal{M}_\sigma\}}^{\{\mathcal{M}_\sigma\}} (\mathbb R^n)$, 
for every $\sigma>1.$

\par

Notice that Gelfand-Shilov spaces related to a 
weight matrix $\mathcal{M} =\{ M ^{(\lambda)} :\lambda>0\}$, when the weight sequence  
$M = ( M_{\alpha})_{\alpha\in \mathbb N_0 ^n}$  satisfies suitable conditions,
are usually given by
\begin{multline} \label{eq:commom-GS-with-geom-factor-Rou}
    \mathscr S_{\{\mathcal{M} \}}^{\{\mathcal{M} \}} (\mathbb R^n)  =
   \Big \{ \varphi \in C^{\infty} (\mathbb R^n) \colon
    (\exists \lambda >0) \;(\exists h>0)\;(\exists C >0) \\
   \sup_{x} |x^\alpha \partial^\beta \varphi(x)|\leq 
   C h^{|\alpha + \beta|} M^{\lambda}_\alpha M^{\lambda}_\beta,
   \;\; \forall \alpha,\beta\in\mathbb N_0^n
   \Big\},
\end{multline}
and
\begin{multline} \label{eq:commom-GS-with-geom-factor}
    \mathscr S_{(\mathcal{M} )}^{(\mathcal{M} )} (\mathbb R^n)  =
    \Big\{ \varphi \in C^{\infty} (\mathbb R^n) \colon
    (\forall \lambda >0) \;(\forall h>0)\;(\exists C = C_{\lambda, h}>0) \\
   \sup_{x} |x^\alpha \partial^\beta \varphi(x)|\leq 
   C h^{|\alpha + \beta|} M^{\lambda}_\alpha M^{\lambda}_\beta,
   \;\; \forall \alpha,\beta\in\mathbb N_0^n
   \Big\},
\end{multline}
see e.g. \cite[(2.16)]{BJOS-2021}. On the other hand, the ''geometric growth factor'' 
$h^{|\alpha + \beta|}$  
in  \eqref{eq:commom-GS-with-geom-factor-Rou} and \eqref{eq:commom-GS-with-geom-factor}
is absent in the  definition of corresponding
Gelfand-Shilov spaces in \cite[(3.1)]{DNV2021}.

Let us show that for Gelfand-Shilov spaces in Definition
\ref{Def:extGS} the expected presence of "non-standard growth factor"
$h^{|\alpha|^\sigma + |\beta|^\sigma}$  
related to
the weight matrix  $\mathcal{M}_\sigma$ given by
\eqref{eq:the-weight-matrix} is irrelevant.

In the proof of Theorem \ref{thm:with-or-without-h} 
we will use the following result, \cite[Lemma 2.3]{PTT-01}. 

\begin{Lemma} \label{lm1}
\begin{equation} \label{eq:lema-from-ptt}
 \sup_{\rho >0}  \frac{h^{\rho^\sigma}}{\rho^{\tau \rho ^\sigma}} 
 = e^{\frac{\tau}{\sigma e} h^{\frac{\sigma}{\tau}}}, \qquad h>0,
  \end{equation}
for any $\tau > 0$ and $\sigma >1$.
  \end{Lemma}

\begin{proof} We provide the proof, correcting a typographical error from \cite{PTT-01}.
Define $f(\rho)=\frac{h^{\rho^{\sigma}}}{\rho^{\tau \rho^{\sigma}}}$, $\rho>0$. Differentiating $\ln f(\rho),$ we obtain
$(\ln f(\rho))'= \rho^{\sigma-1}(\sigma\ln h-\tau\sigma \ln \rho -\tau)$, $\rho>0.$
For $\rho_0:=h^{\frac{1}{\tau}}e^{-1/\sigma},$ we have
$\max_{\rho>0}\ln f(\rho)= \ln f(\rho_0)= \frac{\tau}{e\sigma} h^{\frac{\sigma}{\tau}}$,
which proves the claim.
\end{proof}
  
  Lemma \ref{lm1} implies
that for any given $h > 0$ we have
\begin{equation}\label{jede h}
    h ^{|\alpha|^\sigma}  M^{\frac{\tau}{2}, \sigma} _{|\alpha|} \leq
 C   M^{\tau, \sigma} _{|\alpha|}, \qquad \forall \alpha \in \mathbb{N}_0 ^n,
 \end{equation}
where the constant $C>0$ depends on $\tau$, $\sigma$ and $h$.

We also introduce the notation:
\begin{multline} \label{eq:GS-with-factor}
   \mathscr S_{(\mathcal{M}_\sigma ) }^{ (\mathcal{M}_\sigma )} (\mathbb R^n)  =
    \Big\{ \varphi \in C^{\infty} (\mathbb R^n) \colon
    (\forall \tau>0) \;(\forall h>0)\;(\exists C = C_{\tau, h}>0) \\
   \sup_{x} |x^\alpha \partial^\beta \varphi(x)|\leq 
   C h^{|\alpha|^\sigma + |\beta|^\sigma } M^{\tau, \sigma}_{|\alpha|} 
   M^{\tau, \sigma}_{|\beta|},
   \;\; \forall \alpha,\beta\in\mathbb N_0^n
   \Big\},
\end{multline}
and the spaces 
$ \mathscr S_{ \{\mathcal{M}_\sigma \}} ^{ \{\mathcal{M}_\sigma \}} (\mathbb R^n) $,
$ \mathscr S^{ [\mathcal{M}_\sigma ]} (\mathbb R^n) $,
and $ \mathscr S_{ [\mathcal{M}_\sigma ]} (\mathbb R^n) $
are introduced in a similar fashion.

\begin{Theorem} \label{thm:with-or-without-h}
Let $\sigma>1$, and let the weight matrix $\mathcal{M}_\sigma$ be given by 
\eqref{eq:the-weight-matrix}, 
$ M_{p}^{\tau,\sigma}=p^{\tau p^\sigma}$, $ p\in \mathbb{N}$, $ M_0^{\tau,\sigma}=1$.
Then 
\begin{equation}
\mathcal S_{[\mathcal{M}_\sigma ] }^{ [\mathcal{M}_\sigma]} (\mathbb R^n)
= \mathscr S_{[\mathcal{M}_\sigma ] }^{ [\mathcal{M}_\sigma ]} (\mathbb R^n),    
\end{equation}
$\mathcal S^{ [\mathcal{M}_\sigma]} (\mathbb R^n)
= \mathscr S^{ [\mathcal{M}_\sigma ]} (\mathbb R^n),$
and 
$\mathcal S_{[\mathcal{M}_\sigma ] } (\mathbb R^n)
= \mathscr S_{[\mathcal{M}_\sigma ] } (\mathbb R^n).$
\end{Theorem}

\begin{proof}
We prove $ \mathcal S_{(\mathcal{M}_\sigma ) }^{ (\mathcal{M}_\sigma)} (\mathbb R^n)
= \mathscr S_{(\mathcal{M}_\sigma ) }^{ (\mathcal{M}_\sigma )} (\mathbb R^n)$,
and leave the other cases to the reader.

Let us first show that 
$ \varphi \in \mathscr S_{(\mathcal{M}_\sigma ) }^{ (\mathcal{M}_\sigma )} (\mathbb R^n)$
if and only if
\begin{equation} \label{eq:GS-with-exists}
  (\forall \tau>0) \;(\exists h_0>0)\;(\exists C >0) \;\;\;
   \sup_{x} |x^\alpha \partial^\beta \varphi(x)|\leq 
   C h_0 ^{|\alpha|^\sigma + |\beta|^\sigma } M^{\tau, \sigma}_{|\alpha|} 
   M^{\tau, \sigma}_{|\beta|},
  \end{equation}
for all $  \alpha,\beta\in\mathbb N_0^n$.
The implication \eqref{eq:GS-with-factor} $ \Rightarrow$
 \eqref{eq:GS-with-exists} trivially holds. 
 
 For the opposite inclusion we invoke \eqref{eq:lema-from-ptt}, wherefrom   
 for any given $\tau > 0$ and $h>0$ we have
\begin{eqnarray*}
h_0 ^{|\alpha|^\sigma} M^{\frac{\tau}{2}, \sigma} _{|\alpha|} & = &
\left ( \frac{h_0 }{h} \right ) ^{|\alpha|^\sigma} h ^{|\alpha|^\sigma} 
\frac{ M^{\tau, \sigma} _{|\alpha|}}{ M^{\frac{\tau}{2}, \sigma} _{|\alpha|}}
\leq \left(  \sup_{\alpha >0} 
\frac{( h_0 /h ) ^{|\alpha|^\sigma}}{M^{\frac{\tau}{2}, \sigma} _{|\alpha|}} \right)
\cdot
h ^{|\alpha|^\sigma}  M^{\tau, \sigma} _{|\alpha|} \\
& = & e^{\frac{\tau}{2\sigma e}  ( h_0 /h ) ^{\frac{2\sigma}{\tau}}}  
h ^{|\alpha|^\sigma}  M^{\tau, \sigma} _{|\alpha|} 
= C_{\tau, h} h ^{|\alpha|^\sigma}  M^{\tau, \sigma} _{|\alpha|},     
\end{eqnarray*}
and similarly
$$ h_0 ^{|\beta|^\sigma} M^{\frac{\tau}{2}, \sigma} _{|\beta|} \leq
 C_{\tau,h} h ^{|\beta|^\sigma}  M^{\tau, \sigma} _{|\beta|}.
$$
Consequently,
$$
 \sup_{x} |x^\alpha \partial^\beta \varphi(x)|\leq
  C_{\tau, h} h ^{|\alpha|^\sigma + |\beta|^\sigma}  M^{\tau, \sigma} _{|\alpha|} 
  M^{\tau, \sigma} _{|\beta|},
$$
i.e.  \eqref{eq:GS-with-exists}  $ \Rightarrow$ \eqref{eq:GS-with-factor}.
Thus, it remains to prove that $ \varphi \in \mathcal S_{(\mathcal{M}_\sigma ) }^{ (\mathcal{M}_\sigma)} (\mathbb R^n)$ if and only if \eqref{eq:GS-with-exists} holds.
Of course, $ \varphi \in \mathcal S_{(\mathcal{M}_\sigma ) }^{ (\mathcal{M}_\sigma)} (\mathbb R^n)$ implies  \eqref{eq:GS-with-exists} (take $h_0 \geq 1 $ in 
\eqref{eq:GS-with-exists}). Next, we
assume that \eqref{eq:GS-with-exists} holds and use
\eqref{jede h} to conclude that
$ \varphi \in \mathcal S_{(\mathcal{M}_\sigma ) }^{ (\mathcal{M}_\sigma)} (\mathbb R^n)$.
\end{proof}

\begin{Remark} \label{rem:alfa-beta}
By  $\widetilde{(M.2)}$, \eqref{M.1'}, and Theorem \ref{thm:with-or-without-h}
it follows that the product $M^{\tau, \sigma} _{|\alpha|} 
  M^{\tau, \sigma} _{|\beta|}$ in the definition of 
$\mathcal S_{[\mathcal{M}_\sigma ] }^{ [\mathcal{M}_\sigma]} (\mathbb R^n)$ can be replaced by $M^{\tau, \sigma} _{|\alpha + \beta|}$ yielding the same spaces.
 \end{Remark}

Next we show that instead of the sup norm one can use the $L^2$-norm
in Definition \ref{Def:extGS}. This result will be used in Section \ref{Sec3}.
We refer to \cite[Theorem 1]{Pil89} or 
\cite[Theorem 6.1.6]{Nicola-Rodino}
for Gelfand-Shilov spaces related to Gevrey sequences $ M_{\alpha} = \alpha ! ^t$, $\alpha\in \mathbb N_0 ^n$, $ t>0$.

We introduce the following notation.
\begin{align}
\varphi\in \mathcal  S_{\{\mathcal{M}_\sigma\},2}^{\{\mathcal{M}_\sigma\}} (\mathbb R^n) \quad \Leftrightarrow \quad & (\exists \tau>0)\;(\exists C>0)\;(\forall \alpha,
\beta\in\mathbb N_0^n) \nonumber \\
   &\|x^\alpha \partial^\beta \varphi(x)\|_{L^2}
  \leq C M^{\tau,\sigma}_{|\alpha|}     M^{\tau,\sigma}_{|\beta|};
  \label{eq:L2Roumieu}\\
\varphi\in \mathcal S_{(\mathcal{M}_\sigma),2}^{(\mathcal{M}_\sigma)} (\mathbb R^n) \quad \Leftrightarrow \quad & (\forall \tau>0) (\exists C>0)(\forall \alpha,\beta\in\mathbb N_0^n) 
\nonumber \\
   &\|x^\alpha \partial^\beta \varphi(x)\|_{L^2}\leq C
   M^{\tau,\sigma}_{|\alpha|}    M^{\tau,\sigma}_{|\beta|},
   \label{eq:L2Beurling}
\end{align}
where $\| \cdot \|_{L^2}$ denotes the usual $L^2$-norm,
and $\mathcal S_{[\mathcal{M}_\sigma],2}^{[\mathcal{M}_\sigma]} $ is either
$ \mathcal  S_{\{\mathcal{M}_\sigma\},2}^{\{\mathcal{M}_\sigma\}}$
or $ \mathcal S_{(\mathcal{M}_\sigma),2}^{(\mathcal{M}_\sigma)}$.

\begin{Theorem} \label{thm:L2-and-Linfty}
Let $\sigma>1$, and let the weight matrix $\mathcal{M}_\sigma$ be given by 
\eqref{eq:the-weight-matrix}, 
$ M_{p}^{\tau,\sigma}=p^{\tau p^\sigma}$, $ p\in \mathbb{N}$, $ M_0^{\tau,\sigma}=1$.
Then 
\begin{equation}
\mathcal S_{[\mathcal{M}_\sigma ] }^{ [\mathcal{M}_\sigma]} (\mathbb R^n)
=  \mathcal S_{[\mathcal{M}_\sigma],2}^{[\mathcal{M}_\sigma]} (\mathbb R^n).    \label{eq:L2-and-Linfty}
\end{equation}
\end{Theorem}

\begin{proof} 
We prove the Beurling case, i.e.
$ \mathcal S_{(\mathcal{M}_\sigma ) }^{ (\mathcal{M}_\sigma)} (\mathbb R^n)
= \mathcal S_{(\mathcal{M}_\sigma ),2 }^{ (\mathcal{M}_\sigma )} (\mathbb R^n)$,
and leave the Roumieu case to the reader.

For the inclusion 
$ \mathcal S_{(\mathcal{M}_\sigma ),2 }^{ (\mathcal{M}_\sigma )} (\mathbb R^n) 
\subset \mathcal S_{(\mathcal{M}_\sigma ) }^{ (\mathcal{M}_\sigma)} (\mathbb R^n)
$, we invoke the Sobolev embedding 
$ H^s (\mathbb R^n) \hookrightarrow L^\infty (\mathbb R^n)$, $s> n/2$,
which implies
\begin{equation} \label{eq:sobolev-emb}
\sup_{x} |x^\alpha \partial^\beta \varphi(x)|\leq
\tilde C \| x^\alpha \partial^\beta \varphi(x) \|_{H^s} \leq
C \sum_{|\gamma| \leq s} \| \partial^\gamma (x^\alpha \partial^\beta \varphi(x) )\|_{L^2},
\end{equation}
for some $C, \tilde C >0$,
since the Sobolev norm
$\| f(x) \|_{H^s} = \| (1+|\omega|^2)^{s/2} \hat f(\omega) \|_{H^s}
$
is equivalent to
$ \sum_{|\gamma| \leq s} \| \partial^\gamma f(x)\|_{L^2}$.

\par

By the Leibniz formula 
\begin{equation} \label{eq:several-terms}
\sum_{|\gamma| \leq s} \|\partial^\gamma (x^\alpha \partial^\beta \varphi(x) )\|_{L^2}  \leq
\sum_{|\gamma| \leq s}
\sum_{\delta \leq \min \{ \gamma, \alpha \}}  
\binom{\gamma}{\delta} \binom{\alpha}{\delta} \delta!
\| x^{\alpha - \delta} \partial ^{\beta + \gamma - \delta}  \varphi(x)\|_{L^2}. \\
\end{equation}

Now $ \varphi \in  \mathcal S_{(\mathcal{M}_\sigma ),2 }^{ (\mathcal{M}_\sigma )} (\mathbb R^n) $ implies
\begin{align}
\| x^{\alpha - \delta} \partial ^{\beta + \gamma - \delta}  \varphi(x)\|_{L^2} 
&\leq 
C_1  M^{\tau,\sigma}_{|\alpha - \delta|}    M^{\tau,\sigma}_{|\beta + \gamma - \delta|} \nonumber
\\
&\leq 
C_1  M^{\tau,\sigma}_{|\alpha| }    M^{\tau,\sigma}_{|\beta + s|} \nonumber \\
&\leq 
C_1  M^{\tau,\sigma}_{|\alpha |}  C_2 ^{s |\beta +s|^\sigma}  M^{\tau,\sigma}_{|\beta|} \nonumber \\
& \leq
C_1  M^{\tau,\sigma}_{|\alpha |}  C_2 ^{s 2^{\sigma - 1}(|\beta |^\sigma + s^\sigma)}  M^{\tau,\sigma}_{|\beta|} \nonumber \\
& \leq
C_{3}  M^{\tau,\sigma}_{|\alpha |}  h^{|\beta |^\sigma }  M^{\tau,\sigma}_{|\beta|}, 
\label{eq:estimates}
\end{align}
for some $h>0$, $C_3 >0$ (which depends on $\tau, \sigma$ and the dimension $n$),
where we used inequality
\begin{equation} \label{eq:inequ-powers-of-sigma}
(p+q)^\sigma 
\leq 2^{\sigma - 1}(p^\sigma +q^\sigma), \quad  p,q>0, \quad \sigma>1.    
\end{equation}

Since the number of terms in the sums in \eqref{eq:several-terms} 
can be estimated by an integer independent on $ \alpha$,
by \eqref{eq:sobolev-emb}, \eqref{eq:several-terms}, \eqref{eq:estimates},  and 
Theorem \ref{thm:with-or-without-h} it follows that 
$\varphi \in \mathcal S_{(\mathcal{M}_\sigma ) }^{ (\mathcal{M}_\sigma)} (\mathbb R^n)$.

It remains to prove the opposite inclusion,
$ \mathcal S_{(\mathcal{M}_\sigma ) }^{ (\mathcal{M}_\sigma)} (\mathbb R^n)
\subset 
\mathcal S_{(\mathcal{M}_\sigma ),2 }^{ (\mathcal{M}_\sigma )} (\mathbb R^n).
$

Let there be given $s > n/4$, so that $ C_s = \| (1+|x|^2)^{-s} \|_{L^2} <\infty$. By using
$$
(1+|x|^2)^{s} = \sum_{|\gamma| \leq s} 
C_\gamma | x^{2\gamma} |, \qquad x \in \mathbb R^n,
$$
where the constants $C_\gamma$ can be estimated in terms of $s$, and we get
\begin{align}
\| x^{\alpha } \partial ^{\beta }  \varphi(x)\|_{L^2} 
& \leq C_s  \sup_{x} (1+|x|^2)^{s} |x^\alpha \partial^\beta \varphi(x)| \nonumber \\
& = C_s  \sum_{|\gamma| \leq s} 
C_\gamma  \sup_{x} | x^{\alpha + 2\gamma} \partial^\beta \varphi(x)| 
\nonumber \\
& \leq C_1  \sum_{|\gamma| \leq s} 
 M^{\tau,\sigma}_{|\alpha + 2\gamma|}   M^{\tau,\sigma}_{|\beta|} \nonumber \\
& \leq C_1  \sum_{|\gamma| \leq s} C_2 ^{(2\gamma) 2^{\sigma - 1}(|\alpha |^\sigma + (2\gamma) ^\sigma)}
M^{\tau,\sigma}_{|\alpha |}   M^{\tau,\sigma}_{|\beta|} 
\nonumber \\
& \leq C_1  C_2 ^{ 2^{\sigma - 1}(2s) ^{\sigma+1}}
\sum_{|\gamma| \leq s} h^{|\alpha |^\sigma }
M^{\tau,\sigma}_{|\alpha |}   M^{\tau,\sigma}_{|\beta|}  
\nonumber \\
& \leq C
h^{|\alpha |^\sigma }
M^{\tau,\sigma}_{|\alpha |}   M^{\tau,\sigma}_{|\beta|},
\label{eq:Linfty-implies-L2}
\end{align}
for some $C, C_1, C_2 >0$ which depends on the dimension $n$ and $\sigma>0$.

Thus if 
$\varphi \in 
\mathcal S_{(\mathcal{M}_\sigma ) }^{ (\mathcal{M}_\sigma)} (\mathbb R^n)$, 
by \eqref{eq:Linfty-implies-L2} and  Theorem \ref{thm:with-or-without-h} we get
$\varphi \in \mathcal S_{(\mathcal{M}_\sigma ),2 }^{ (\mathcal{M}_\sigma)} (\mathbb R^n)$.
\end{proof}

\begin{Remark}  \label{rem:Lp-and-Linfty}
If $L^2$-norm in \eqref{eq:L2Roumieu} and \eqref{eq:L2Beurling} 
is replaced by $L^p$-norm, and we denote the corresponding 
spaces by $\mathcal S_{[\mathcal{M}_\sigma],p}^{[\mathcal{M}_\sigma]} $, $ 1\leq p < \infty$, then 
\eqref{eq:L2-and-Linfty} in Theorem \ref{thm:L2-and-Linfty}
extends to
\begin{equation}
\mathcal S_{[\mathcal{M}_\sigma ] }^{ [\mathcal{M}_\sigma]} (\mathbb R^n)
=  \mathcal S_{[\mathcal{M}_\sigma],p}^{[\mathcal{M}_\sigma]} (\mathbb R^n), \quad 1\leq p < \infty.    \label{eq:Lp-and-Linfty}
\end{equation}
This can be proved by using the $L^p$-Sobolev embedding, the H\"older inequality, and a modification of  the proof of Theorem \ref{thm:L2-and-Linfty}. We refer to \cite{Pil89} where Gelfand-Shilov spaces related to Gevrey sequences are considered.
\end{Remark}

It can be easily verified that
$\mathcal S^{ [\mathcal{M}_\sigma]} (\mathbb R^n)$,
$\mathcal S_{ [\mathcal{M}_\sigma]} (\mathbb R^n)$, and
$\mathcal S_{[\mathcal{M}_\sigma ] }^{ [\mathcal{M}_\sigma]} (\mathbb R^n)$ are 
closed under differentiation and multiplication by polynomials.


\subsection{Nuclearity}

In this subsection we prove that spaces $\mathcal S_{[\mathcal{M}_\sigma ] }^{ [\mathcal{M}_\sigma]} (\mathbb R^n)$ are nuclear.
This is done by using general theory given in \cite{DNV2021}.
We perform nontrivial calculations to show 
that the general result \cite[Theorem 5.1]{DNV2021} 
can be utilized to prove the nuclearity of 
$\mathcal S_{[\mathcal{M}_\sigma ] }^{ [\mathcal{M}_\sigma]} (\mathbb R^n)$.

Recall that the associated function to the sequence $M^{\tau,\sigma}_p$ is given by 
\begin{equation}
    \label{eq:associated-function}
T_{\tau,\sigma}(x)=\sup_{p\in \mathbb{N}_0}\ln \frac{x^p}{M^{\tau, \sigma}_p}, \qquad
x>0. 
\end{equation}
Then by \cite[Proposition 2]{TT0}
it follows that
\begin{multline}
\label{nejednakostAsocirana}
\frac{1}{A}  \tau^{-\frac{1}{\sigma-1}} \frac{\ln^{\frac{\sigma}{\sigma-1}}(|t|)}{W^{\frac{1}{\sigma-1}}(\ln(|t|))}-B
 \leq T_{\tau,\sigma}(|t|) 
\\
 \leq A  \tau^{-\frac{1}{\sigma-1}} 
\frac{\ln^{\frac{\sigma}{\sigma-1}}(|t|)}{W^{\frac{1}{\sigma-1}}(\ln(|t|))}+B, \quad |t| > 1,
\end{multline}
for suitable constants $A>1$ and $B>0$.

Let us define the family of functions $\mathcal{W}_{\sigma} $ as follows:


\begin{equation}
\mathcal{W}_{\sigma}=\left\{\left(w^{\tau}(x)=e^{T_{\tau,\sigma}(|x|)}\right)_{x\in \mathbb{R}^n}\colon \tau>0\right\}.
\end{equation} 

Then $\varphi\in C^\infty (\mathbb R^n) $ belongs to
$ \mathcal S_{\{\mathcal{W}_\sigma \}}^{\{\mathcal{M}_\sigma \}} (\mathbb R^n) 
$
if
\begin{equation}
\label{DefinicijaWeightNiz}
 (\exists \tau>0)) (\exists C>0)(\forall \beta\in\mathbb N_0^n) 
\quad   \sup_{x}\left| e^{T_{\tau,\sigma}(|x|)} \partial^\beta \varphi(x)\right|\leq C
     M^{\tau,\sigma}_{|\beta|};
\end{equation}
and $\displaystyle \varphi\in \mathcal S_{(\mathcal{W}_\sigma)}^{(\mathcal{M}_\sigma)} (\mathbb R^n) 
$ if 
\begin{equation}
\label{DefinicijaWeightNiz-2}
 (\forall \tau>0)) (\exists C>0)(\forall \beta\in\mathbb N_0^n) 
\quad
   \sup_{x}\left| e^{T_{\tau,\sigma}(|x|)} \partial^\beta \varphi(x)\right|\leq C
     M^{\tau,\sigma}_{|\beta|}.
\end{equation}   

Again, we use 
$\mathcal S_{[\mathcal{W}_\sigma ] }^{ [\mathcal{M}_\sigma]}$ to denote $ \mathcal S_{\{\mathcal{W}_\sigma \}}^{\{\mathcal{M}_\sigma \}} (\mathbb R^n) 
$ 
or
$\mathcal S_{(\mathcal{W}_\sigma)}^{(\mathcal{M}_\sigma)} (\mathbb R^n). 
$

Now we can prove the main result of this section.

\begin{Theorem}
Fix $\sigma>1$ and let $\mathcal{M}_\sigma$ be a weight matrix given with \eqref{eq:the-weight-matrix}. Then: 
\begin{itemize}
\item[1.] $\mathcal S_{[\mathcal{M}_\sigma ] }^{ [\mathcal{M}_\sigma]} (\mathbb R^n)=\mathcal S_{[\mathcal{W}_\sigma ] }^{ [\mathcal{M}_\sigma]} (\mathbb R^n)$,
\item[2.]  $\mathcal S_{[\mathcal{W}_\sigma ] }^{ [\mathcal{M}_\sigma]} (\mathbb R^n)$ is a nuclear space.
\end{itemize}

\end{Theorem}

\begin{proof}
We prove the Roumieu case, and the Beurling case follows by using similar arguments.

{\em 1.}
Since the sequence $(M^{\tau,\sigma}_p)_{p\in\mathbb{N}_0}$ given by 
\eqref{eq:the-sequence} satisfies $(M.1)$, 
by \cite[Proposition 3.2]{Komatsuultra1}
it follows that
\begin{equation}
\label{KomatsuM1}
M^{\tau,\sigma}_p=\sup_{|x|> 0}\frac{|x|^p}{e^{T_{\tau,\sigma}(|x|)}},\quad p\in\mathbb{N}
\end{equation}
for all $\tau>0$ and $\sigma>1$.

Let $\varphi\in \mathcal S_{\{\mathcal{M}_\sigma \} }^{ \{\mathcal{M}_\sigma\}}  (\mathbb R^n)$. Then there exist $\tau_0, C>0$ such that
$$
\sup_{x \in  \mathbb R^n}|x^\alpha \partial^\beta \varphi(x)|\leq C M^{\tau_0,\sigma}_{|\alpha|} 
   M^{\tau_0,\sigma}_{|\beta|},
   $$ and by the simple inequality $(1/\sqrt{n})^{|\alpha|}|x|^{|\alpha|}\leq |x^{\alpha}|$, we get
$$
|\partial^\beta \varphi(x)|\leq C 
\frac{\sqrt{n}^{|\alpha|} M^{\tau_0,\sigma}_{|\alpha|}}{|x|^{|\alpha|}} 
   M^{\tau_0,\sigma}_{|\beta|}= C \frac{\sqrt{n}^{|\alpha|}}{M^{\tau_0,\sigma}_{|\alpha|}} \frac{ M^{2\tau_0,\sigma}_{|\alpha|}}{|x|^{|\alpha|}} M^{\tau_0,\sigma}_{|\beta|}\leq C' M^{\tau_0,\sigma}_{|\beta|}  \frac{ M^{2\tau_0,\sigma}_{|\alpha|}}{|x|^{|\alpha|}}, $$ for $\alpha,\beta \in \mathbb{N}_0^n,$ 
   $|x|\not=0$, where $C'=C e^{T_{\tau_0,\sigma}(\sqrt{n})}$. By \eqref{KomatsuM1} we     obtain 
$$
 \sup_{x}| e^{T_{\tau,\sigma}(|x|)} \partial^\beta \varphi(x)|\leq 
  C' M^{\tau_0,\sigma}_{|\beta|}  \sup_{x} e^{T_{\tau,\sigma}(|x|)} \frac{ M^{2\tau_0,\sigma}_{|\alpha|}}{|x|^{|\alpha|}} \leq
 C'     M^{\tau,\sigma}_{|\beta|},
$$
i.e. \eqref{DefinicijaWeightNiz} holds with $\tau=2\tau_0$.

\par

For the opposite direction, we note that
$\varphi\in 
\mathcal S_{\{\mathcal{W}_\sigma \} }^{ \{\mathcal{M}_\sigma\}} (\mathbb R^n)$ i.e.
\eqref{DefinicijaWeightNiz}, together with  \eqref{KomatsuM1} and $|x^{\alpha}|\leq |x|^{|\alpha|}$,
implies
$$
|x^{\alpha} \partial^\beta \varphi(x)|
\leq 
C M^{\tau_0,\sigma}_{|\beta|} |x|^{|\alpha|}
e^{-T_{\tau_0, \sigma}(|x|)} 
\leq 
C  M^{\tau_0,\sigma}_{|\beta|}
M^{\tau_0,\sigma}_{|\alpha|}
$$ 
for $\alpha, \beta \in \mathbb{N}_0^n$.
Thus  $\varphi\in \mathcal S_{\{\mathcal{M}_\sigma \} }^{ \{\mathcal{M}_\sigma\}}  (\mathbb R^n)$, which proves {\em 1.}


{\em 2.} By \cite[Theorem 5.1]{DNV2021},  $\mathcal S_{[\mathcal{M}_\sigma ] }^{ [\mathcal{M}_\sigma]} (\mathbb R^n)=\mathcal S_{[\mathcal{W}_\sigma ] }^{ [\mathcal{M}_\sigma]} (\mathbb R^n)$ is a nuclear space if $\mathcal{M}_{\sigma}$ and $\mathcal{W}_{\sigma}$ satisfy the following conditions:
\begin{align*}
\{L\}&\quad (\forall h>0)(\forall \tau>0)(\exists \tau_0,C>0)(\forall \alpha \in \mathbb N^{n})\\ &  h^{|\alpha|} M^{\tau,\sigma}_{|\alpha|}\leq C M^{\tau_0,\sigma}_{|\alpha|} ,\\
\{\mathcal{M}.2\}'&\quad (\forall \tau>0)(\exists \tau_0>0)(\exists h,C>0)(\forall \alpha \in \mathbb N^{n})\\ &  M^{\tau,\sigma}_{|\alpha|+1}\leq C h^{|\alpha|} M^{\tau_0,\sigma}_{|\alpha|}, \\
\{wM\}&\quad (\forall \tau>0) (\exists \tau_0>0)(\exists C>0) (\forall x\in \mathbb{R}^n)\\\ & \sup_{|y|\leq 1} w^{\tau_0}(x+y)\leq C w^{\tau}(x),\\
\{M\}& \quad (\forall \tau_1,\tau_2>0) (\exists \tau_0>0)(\exists C>0) (\forall x,y\in \mathbb{R}^n)\, \\ & w^{\tau_0}(x+y)\leq C w^{\tau_1}(x)w^{\tau_2}(y),\\
\{N\}&\quad (\forall \tau>0)(\exists \tau_0>0) \\ & w^{\tau_0}/w^{\tau}\in L^1,
\end{align*} where $w^{\tau}(x)=e^{T_{\tau,\sigma}(|x|)}$. Note that $\{M\}$ implies $\{wM\}$.

We note that $\{L\}$ follows from Lemma \ref{lm1}, i.e. from \eqref{jede h}.


For $\{\mathcal{M}.2\}'$ we take arbitrary $\tau>0$. Then by $\widetilde{(M.2)}'$ we have
$$
M^{\tau,\sigma}_{|\alpha|+1}\leq A^{p^{\sigma}+1} M^{\tau,\sigma}_{|\alpha|}=  A \frac{A^{p^{\sigma}}}{M^{\tau,\sigma}_{|\alpha|}} M^{2\tau,\sigma}_{|\alpha|}\leq C M^{\tau_0,\sigma}_{|\alpha|},
$$ for $\tau_0= 2\tau$ where we used \eqref{eq:lema-from-ptt}.

To verify conditions $\{wM\}$ and $\{N\}$ put 
$$ \displaystyle 
\omega_{\sigma}(x)= \frac{\ln^{\frac{\sigma}{\sigma-1}}(1+|x|)}{W^{\frac{1}{\sigma-1}}(\ln(1+|x|))}, 
\quad x\in \mathbb{R}^n \setminus 0, 
\quad \omega_{\sigma}(0)=0,
$$   
and consider
$$
\mathcal{W}'_{\sigma}=\left\{\left(\omega^{\lambda}(x)=e^{\frac{1}{\lambda}\omega_\sigma(x)}\right)_{x\in \mathbb{R}^n}\colon \lambda>0\right\}.
$$ 
By \eqref{nejednakostAsocirana} we get 
$$
(\forall \tau>0)(\exists \lambda_1, \lambda_2>0)(\exists C_1, C_2>0)(\forall x\in \mathbb{R}^n)\,\, C_1 \omega^{\lambda_1}(x) \leq w^{\tau}(x)\leq C_2\omega^{\lambda_2}(x),
$$ 
wherefrom it follows that the families $\mathcal{W}_{\sigma}$ and $\mathcal{W}'_{\sigma}$ are equivalent.
Thus, it is sufficient to check that 
$\mathcal{W}'_{\sigma}$ satisfies $\{wM\}$ and $\{N\}$.

Indeed, since $\omega_{\sigma}(x)$ is equivalent to a BMT function (see Remark \ref{RemarkBMT}), 
it satisfies BMT conditions $(\alpha)-(\gamma)$. 
Then, \cite[Lemma 3.5]{DNV2021} implies that 
$\mathcal{W}'_{\sigma}$ satisfies conditions $\{M\}$ and $\{N\}$, and the theorem is proved.
\end{proof}

\section{Fourier transform representation}  \label{Sec3}

The Gelfand-Shilov spaces given by \eqref{Eq:GeShSpDef} with
$t= s > 0$ are Fourier transform invariant, which is a convent property in different applications. In this section we address the Fourier transform invariance of $\mathcal S_{[\mathcal{M}_\sigma]}^{[\mathcal{M}_\sigma]} (\mathbb R^n)$, and their 
characterizations in terms of 
the Fourier transform. 

The Fourier transform 
of $f \in \mathcal S_{[\mathcal{M}_\sigma]}^{[\mathcal{M}_\sigma]} (\mathbb R^n)$ 
is given by
\begin{align*}
    (\mathcal{F}f)(\xi)=\hat f(\xi)=\int_{\mathbb R^n}f(x)e^{-2\pi ix\xi}\;dx,\quad \xi\in \mathbb R^n,
\end{align*}
and the corresponding inverse Fourier transform  is defined as
\begin{align*}
    (\mathcal{F}^{-1}f)(x)=\int_{\mathbb R^n}f(\xi)e^{2\pi ix\xi}\;d\xi,\quad x\in \mathbb R^n.
\end{align*}

The invariance of classical Gelfand-Shilov spaces 
under the Fourier transform is given already in \cite{GelfandShilov}, see also 
a more recent source \cite[Chapter 6]{Nicola-Rodino}.
The Fourier transform invariance for the 
Gelfand-Shilov type spaces related to the extended Gevrey regularity can be stated as follows.

\begin{Theorem}\label{Thm:invariant F}
    Let $\sigma>1$, and let the weight matrix $\mathcal{M}_\sigma$ be given by 
\eqref{eq:the-weight-matrix}, 
$ M_{p}^{\tau,\sigma}=p^{\tau p^\sigma}$, $ p\in \mathbb{N}$, $ M_0^{\tau,\sigma}=1$.
    
Then $ \mathcal{F} ( \mathcal S_{[\mathcal{M}_\sigma]}^{[\mathcal{M}_\sigma]} ) (\mathbb R^n) = 
\mathcal S_{[\mathcal{M}_\sigma]}^{[\mathcal{M}_\sigma]} (\mathbb R^n).$ 
\end{Theorem}
    
\begin{proof}
We give the proof  for the Beurling case, 
since the Roumieu case can be proved in a similar way. 
  
  Let 
  $\varphi\in \mathcal S_{(\mathcal{M}_\sigma)}^{(\mathcal{M}_\sigma)} (\mathbb R^n)$ and let there  be given $\tau>0$.
  

   
  From the Leibniz formula and \eqref{M.1'} we obtain
   \begin{align*}
       |\partial^\beta(x^\alpha \varphi(x))|& =\left|\sum_{0\leq\delta\leq\min\{\alpha,\beta\}} \binom{\alpha}{\delta}\binom{\beta}{\delta}\delta!\,x^{\alpha-\delta}\partial^{\beta-\delta}\varphi(x) \right |\\
       &\leq \sum_{0\leq\delta\leq\min\{\alpha,\beta\}} \binom{\alpha}{\delta}\binom{\beta}{\delta}\delta!\, C\, M^{\frac{\tau}{2},\sigma}_{|\alpha-\delta|} M^{\frac{\tau}{2},\sigma}_{|\beta-\delta|}\\
       &\leq C\, M^{\frac{\tau}{2},\sigma}_{|\alpha|} M^{\frac{\tau}{2},\sigma}_{|\beta|} \sum_{0\leq\delta\leq\min\{\alpha,\beta\}} \binom{\alpha}{\delta}\binom{\beta}{\delta}\, \frac{\delta!}{(M_{|\delta|}^{\frac{\tau}{2},\sigma})^2} .
\end{align*}
Next, by using the fact  that for any $\tau>0$ there exists $C_1>0$ such that 
\begin{equation}\label{gevr to us}
    \delta!\leq C_1 M_{|\delta|}^{\tau,\sigma},\quad \delta\in \mathbb N_0^n,
\end{equation}
(which follows form e.g. \cite[(4)]{TTZ2024}), we obtain
\begin{align*}
       |\partial^\beta(x^\alpha \varphi(x))| &\leq C\, M^{\frac{\tau}{2},\sigma}_{|\alpha|} M^{\frac{\tau}{2},\sigma}_{|\beta|}\, C' \sum_{0\leq\delta\leq\min\{\alpha,\beta\}} \binom{\alpha}{\delta}\binom{\beta}{\delta}\\
       &\leq C C'\,2^{|\alpha|+|\beta|}\, M^{\frac{\tau}{2},\sigma}_{|\alpha|} M^{\frac{\tau}{2},\sigma}_{|\beta|}\leq C_1  M^{\tau,\sigma}_{|\alpha|} M^{\tau,\sigma}_{|\beta|},\quad \alpha,\beta\in\mathbb N_0^n,
   \end{align*}
   where in the last inequality we applied \eqref{jede h}.
   Similarly, 
   \begin{align*}
       |x|^2|\partial^\beta(x^\alpha \varphi(x))|\leq C''\, 2^{|\alpha|+2+|\beta|}\,M^{\frac{\tau}{2},\sigma}_{|\alpha|+2} M^{\frac{\tau}{2},\sigma}_{|\beta|},\quad \alpha,\beta\in\mathbb N_0^n. 
   \end{align*}
   Applying $\widetilde{ (M.2)'}$ twice and then using \eqref{jede h} again, we obtain
    \begin{align*}
       |x|^2|\partial^\beta(x^\alpha \varphi(x))|&\leq C''\, 2^{|\alpha|+2+|\beta|}\,A^{(|\alpha|+1)^\sigma+|\alpha|^\sigma}\,M^{\frac{\tau}{2},\sigma}_{|\alpha|} M^{\frac{\tau}{2},\sigma}_{|\beta|}\\
       &\leq  C_2  M^{\tau,\sigma}_{|\alpha|} M^{\tau,\sigma}_{|\beta|},\quad \alpha,\beta\in\mathbb N_0^n. 
   \end{align*}
   Finally, 
   \begin{align*}
    |\xi^\beta \mathcal{F}(\varphi)^{(\alpha)}(\xi)|& =\left|\mathcal{F}\left(\partial^\beta(x^\alpha\varphi(x))\right)(\xi)\right|=\left|\int_{\mathbb R^n}\partial^\beta(x^\alpha\varphi(x)) \,e^{-2\pi ix\xi}dx\right|\\
    & \leq \int_{\mathbb R^n}|\partial^\beta(x^\alpha\varphi(x))|\,dx=\int_{\mathbb R^n}\frac{|(1+|x|^2)\,\partial^\beta(x^\alpha\varphi(x))|}{1+|x|^2}dx\\
    &\leq \left(C_1\, M^{\tau,\sigma}_{|\alpha|} M^{\tau,\sigma}_{|\beta|} +  C_2\, M^{\tau,\sigma}_{|\alpha|} M^{\tau,\sigma}_{|\beta|}\right)\int_{\mathbb R^n} \frac{1}{1+|x|^2}dx\\
    & \leq \tilde C M^{\tau,\sigma}_{|\alpha|} M^{\tau,\sigma}_{|\beta|}, \quad \alpha,\beta\in \mathbb N_0^n.
   \end{align*}
So, $\mathcal{F} (\mathcal S_{(\mathcal{M}_\sigma)}^{(\mathcal{M}_\sigma)} (\mathbb R^n)) \subseteq \mathcal S_{(\mathcal{M}_\sigma)}^{(\mathcal{M}_\sigma)} (\mathbb R^n).$ The other inclusion can be obtained 
in the same way since $\mathcal{F}(\mathcal{F}\varphi)(x)=\varphi(-x).$ 
\end{proof}

\begin{Remark}
From the proof of Theorem \ref{Thm:invariant F}, it follows that
$\mathcal{F} (\mathcal S_{[\mathcal{M}_\sigma]})(\mathbb R^n)= \mathcal S^{[\mathcal{M}_\sigma]} (\mathbb R^n)$, and $\mathcal{F} (\mathcal S^{[\mathcal{M}_\sigma]}) (\mathbb R^n)= \mathcal S_{[\mathcal{M}_\sigma]} (\mathbb R^n).$
\end{Remark}


The beautiful symmetric characterizations of 
Gelfand-Shilov spaces given in  \cite{Chung 1996} 
can be formulated in terms of the extended Gevrey regularity as follows. 

\begin{Theorem} \label{Thm:opisFurije}
Let $\sigma>1$,  and let the weight matrix $\mathcal{M}_\sigma$ be given by \eqref{eq:the-weight-matrix}, 
$ M_{p}^{\tau,\sigma}=p^{\tau p^\sigma}$, $ p\in \mathbb{N}$, $ M_0^{\tau,\sigma}=1$.
Then the following conditions are equivalent:
    \begin{enumerate}
        \item $\varphi\in \mathcal  S_{\{\mathcal{M}_\sigma\}}^{\{\mathcal{M}_\sigma\}} (\mathbb R^n);$
        \item $(\exists \tau>0)(\exists C>0)\;(\forall \alpha,\beta\in \mathbb N_0^n)$
        $$\sup_x|x^\alpha \varphi(x)|\leq C M^{\tau,\sigma}_{|\alpha|} \quad \text{and}\quad \sup_x|\partial^\beta \varphi(x)|\leq CM^{\tau,\sigma}_{|\beta|};$$
        \item $(\exists \tau>0)(\exists C>0)\;(\forall \alpha,\beta\in \mathbb N_0^n)$
        $$\sup_x|x^\alpha \varphi(x)|\leq C M^{\tau,\sigma}_{|\alpha|} \quad \text{and}\quad \sup_\xi|\xi^\beta \widehat\varphi(\xi)|\leq CM^{\tau,\sigma}_{|\beta|}.$$
    \end{enumerate}
\end{Theorem}

\begin{proof}
{\em 1.} $ \Rightarrow $ {\em 2.}, 
and {\em 1.} $ \Rightarrow $ {\em 3.}, follows immediately from   Definition \ref{Def:extGS} and Theorem \ref{Thm:invariant F}.

{\em 2.} $ \Rightarrow $ {\em 1.}
   By Theorem \ref{thm:L2-and-Linfty}, we can use $L^2$-norm instead of $L^\infty$-norm, cf. \eqref{eq:L2Roumieu}. Using integration by parts, the Leibniz formula and the Schwarz inequality we obtain that there exist $\tau>0$ and $C>0$ such that 
    \begin{align*}
        \|x^\alpha \partial^\beta\varphi(x)\|^2_{L^2} & =\int_{\mathbb R^n}(x^{2\alpha}\partial^\beta \varphi(x))\partial^\beta \varphi(x) dx\\
        &\leq \sum_{\gamma\leq \min\{2\alpha,\beta\}}\binom{\beta}{\gamma}\binom{2\alpha}{\gamma}\gamma!\|\partial^{2\beta-\gamma}\varphi(x)\|_{L^2}\|x^{2\alpha-\gamma}\varphi(x)\|_{L^2}\\
        &\leq \sum_{\gamma\leq \min\{2\alpha,\beta\}}\binom{\beta}{\gamma}\binom{2\alpha}{\gamma}\gamma!\,C\,M^{\tau,\sigma}_{|2\beta-\gamma|}\,C\,M^{\tau,\sigma}_{|2\alpha-\gamma|},\quad \alpha,\beta\in \mathbb N_0^n.
        \end{align*}
    Next, applying \eqref{M.1'} we get
    \begin{align*}
        \|x^\alpha \partial^\beta\varphi(x)\|^2_{L^2} &\leq C^2  M^{\tau,\sigma}_{|2\alpha|}M^{\tau,\sigma}_{|2\beta|} \sum_{\gamma\leq \min\{2\alpha,\beta\}}\binom{\beta}{\gamma}\binom{2\alpha}{\gamma}\frac{\gamma!}{(M^{\tau,\sigma}_{|\gamma|})^2}.
    \end{align*}
Then we use \eqref{gevr to us}, i.e. $\exists C'>0$ such that for any $\alpha,\beta\in \mathbb N_0^n$,
    \begin{align*}
        \sum_{\gamma\leq \min\{2\alpha,\beta\}}\binom{\beta}{\gamma}\binom{2\alpha}{\gamma}\frac{\gamma!}{(M^{\tau,\sigma}_{|\gamma|})^2}&\leq C' \sum_{\gamma\leq \min\{2\alpha,\beta\}}\binom{\beta}{\gamma}\binom{2\alpha}{\gamma}\\
        &\leq C' 2^{|2\alpha|+|\beta|} \leq C' 2^{|2\alpha|^\sigma+|2\beta|^\sigma},
    \end{align*}
so that
     \begin{align*}
        \|x^\alpha \partial^\beta\varphi(x)\|^2_{L^2} &\leq C^2  \,C' 2^{|2\alpha|^\sigma+|2\beta|^\sigma}\,M^{\tau,\sigma}_{|2\alpha|}M^{\tau,\sigma}_{|2\beta|}, \quad \alpha,\beta\in \mathbb N_0^n.
    \end{align*}
    Now, using $\widetilde{(M.2)},$  we obtain that there exists $A>0$ such that
    \begin{align*}
        \|x^\alpha \partial^\beta\varphi(x)\|^2_{L^2} &\leq C^2\,C' 2^{|2\alpha|^\sigma+|2\beta|^\sigma}\,  A^{2|\alpha|^\sigma} (M_{|\alpha|}^{\tau 2^{\sigma-1},\sigma})^2 A^{2|\beta|^\sigma} (M_{|\beta|}^{\tau 2^{\sigma-1},\sigma})^2 ,
    \\
    &\leq  C^2\,C' (2A)^{|2\alpha|^\sigma+|2\beta|^\sigma} (M_{|\alpha|}^{\tau 2^{\sigma-1},\sigma})^2 (M_{|\beta|}^{\tau 2^{\sigma-1},\sigma})^2, \quad \alpha,\beta\in \mathbb N_0^n.
    \end{align*}
   Finally, by \eqref{jede h} there exists $C''>0$ such that
     \begin{align*}
        \|x^\alpha \partial^\beta\varphi(x)\|_{L^2} &\leq \sqrt{C'} C ((2A)^{2^{\sigma-1}})^{|\alpha|^\sigma+|\beta|^\sigma} M_{|\alpha|}^{\tau 2^{\sigma-1},\sigma} M_{|\beta|}^{\tau 2^{\sigma-1},\sigma} \\
        &\leq \sqrt{C'} C C'' M_{|\alpha|}^{\tau 2^\sigma,\sigma} M_{|\beta|}^{\tau 2^\sigma,\sigma},\quad \alpha,\beta\in \mathbb N_0^n,
    \end{align*}
    which (with $\tilde C=\sqrt{C'} C C''$ and $\tilde \tau=\tau 2^\sigma$) implies \eqref{eq:L2Roumieu}.
    
It remains to prove  {\em 3.} $ \Rightarrow $ {\em 2.},
    
    Let $\tau>0$ and $C>0$ be as in 3. Then for any $\beta \in\mathbb N_0^n,$ the inequality $|\xi^\beta \widehat\varphi(\xi)|\leq CM_{|\beta|}^{\tau,\sigma},$ $\xi \in\mathbb R^n,$ implies 
    $$
    |\widehat\varphi(\xi)|\leq Ce^{-T_{\tau,\sigma}(|\xi|)},\quad \xi \in\mathbb R^n,
    $$ 
    where $T_{\tau,\sigma}$ 
    is the associated function to $(M_{\alpha}^{\tau,\sigma})_{\alpha\in\mathbb N_0^n}$, see 
    \eqref{eq:associated-function}.

By using the fact that
$ \displaystyle 
\int_{\mathbb R^n}e^{-\frac{1}{2}T_{\tau,\sigma}(|\xi|)}d\xi < \infty$ 
(the integrand decreases faster than $|\xi|^{-p}$ for any $p\in\mathbb N_0$), we get
\begin{align*}
    |\partial^\beta \varphi(x)| & \leq \int_{\mathbb R^n}|e^{2\pi ix\xi}\xi^{|\beta|}\widehat\varphi(\xi)|d\xi\\
    &\leq C \int_{\mathbb R^n}|\xi|^{|\beta|}e^{-T_{\tau,\sigma}(|\xi|)}d\xi\\
    & \leq C \sup_{\xi}[|\xi|^{2|\beta|}e^{-T_{\tau,\sigma}(|\xi|)}]^{1/2}\int_{\mathbb R^n}e^{-\frac{1}{2}T_{\tau,\sigma}(|\xi|)}d\xi\\
    & \leq C' \sup_{\xi}[|\xi|^{2|\beta|}e^{-T_{\tau,\sigma}(|\xi|)}]^{1/2}\\
    &\leq C' [M_{|2\beta|}^{\tau,\sigma}]^{1/2},\quad \beta\in\mathbb N_0^n.
\end{align*}
Finally, $\widetilde{(M.2)}$ and \eqref{jede h} yields
\begin{align*}
     |\partial^\beta \varphi(x)| & \leq C'
(A^{2|\beta|^\sigma} (M_{|\beta|}^{\tau 2^{\sigma-1},\sigma})^2)^{1/2} = C'A^{|\beta|^\sigma} M_{|\beta|}^{\tau 2^{\sigma-1},\sigma}\\
& \leq C'C'' M_{|\beta|}^{\tau2^\sigma,\sigma},\quad \beta\in\mathbb N_0^n,
\end{align*}
 which implies {\em 2.} (with $\tilde C= C' C''$ and $\tilde \tau=\tau 2^\sigma$).
\end{proof}

Theorem \ref{Thm:opisFurije} implies the following characterization in terms of the associated functions.

\begin{Corollary}
    Let $\sigma>1$, and let the weight matrix $\mathcal{M}_\sigma$ be given by \eqref{eq:the-weight-matrix}, 
$ M_{p}^{\tau,\sigma}=p^{\tau p^\sigma}$, $ p\in \mathbb{N}$, $ M_0^{\tau,\sigma}=1$.     
    The following conditions are equivalent:
    \begin{enumerate}
\item 
$\varphi\in  \mathcal S_{\{\mathcal{M}_\sigma\}}^{\{\mathcal{M}_\sigma\}} (\mathbb R^n);$
\item there exist $\tau>0$ such that
$$
        \sup_x|\varphi(x)| \exp T_{\tau,\sigma}(|x|)<\infty \quad \text{and}\quad \sup_\xi|\widehat\varphi(\xi)| \exp T_{\tau,\sigma}(|\xi|)<\infty,
$$
where $T_{\tau,\sigma}$ denotes the associated function to $(M_{\alpha}^{\tau,\sigma})_{\alpha\in\mathbb N_0^n}$ given in  
    \eqref{eq:associated-function}.
    \end{enumerate}
\end{Corollary}

A similar theorem can be shown for the Beurling case by following the same proof ideas.

\begin{Theorem} \label{Thm:opisFurijeB}
    Let $\sigma>1,$ and let the weight matrix $\mathcal{M}_\sigma$ be given by \eqref{eq:the-weight-matrix}, 
$ M_{p}^{\tau,\sigma}=p^{\tau p^\sigma}$, $ p\in \mathbb{N}$, $ M_0^{\tau,\sigma}=1$. Furthermore, let $T_{\tau,\sigma}$ be the associated function to $(M_{\alpha}^{\tau,\sigma})_{\alpha\in\mathbb N_0^n}$, see    \eqref{eq:associated-function}.
    Then the  following conditions are equivalent:
    \begin{enumerate}
        \item $\varphi\in \mathcal  S_{(\mathcal{M}_\sigma)}^{(\mathcal{M}_\sigma)} (\mathbb R^n);$
        \item $(\forall \tau>0)(\exists C>0)\;(\forall \alpha,\beta\in \mathbb N_0^n)$
        $$\sup_x|x^\alpha \varphi(x)|\leq C M^{\tau,\sigma}_{|\alpha|} \quad \text{and}\quad \sup_x|\partial^\beta \varphi(x)|\leq CM^{\tau,\sigma}_{|\beta|};$$
        \item $(\forall \tau>0)(\exists C>0)\;(\forall \alpha,\beta\in \mathbb N_0^n)$
        $$\sup_x|x^\alpha \varphi(x)|\leq C M^{\tau,\sigma}_{|\alpha|} \quad \text{and}\quad \sup_\xi|\xi^\beta \widehat\varphi(\xi)|\leq CM^{\tau,\sigma}_{|\beta|};$$
        \item $(\forall \tau>0)$
        $$\sup_x|\varphi(x)| \exp T_{\tau,\sigma}(|x|)<\infty \quad \text{and}\quad \sup_\xi|\widehat\varphi(\xi)| \exp T_{\tau,\sigma}(|\xi|)<\infty.$$
    \end{enumerate}
\end{Theorem}

\begin{Remark}
   Part of the proof of Theorems \ref{Thm:opisFurije} and \ref{Thm:opisFurijeB} (when showing that {\em  2.} 
   implies {\em 1.}, establishes the inclusion $\mathcal  S_{[\mathcal{M}_\sigma]} (\mathbb R^n) \cap \mathcal  S^{[\mathcal{M}_\sigma]} (\mathbb R^n) \subseteq \mathcal  S_{[\mathcal{M}_\sigma]}^{[\mathcal{M}_\sigma]} (\mathbb R^n).$ This, together with \eqref{eq:incl}, this implies that $\mathcal  S_{[\mathcal{M}_\sigma]} \cap \mathcal  S^{[\mathcal{M}_\sigma]} (\mathbb R^n) = \mathcal  S_{[\mathcal{M}_\sigma]}^{[\mathcal{M}_\sigma]} (\mathbb R^n).$
\end{Remark}


\section{Time-frequency representations}  \label{Sec4}

The Fourier transform of $f \in L^2 (\mathbb R^n)$ provides  information about its global frequency content. Different representations of $f$ in phase space or time-frequency plane are used to obtain a localized version of the Fourier transform. Notable examples are the short-time Fourier transform, the Wigner transform and Cohen's class representations such as the Born-Jordan transform.

In this section we turn our attention to the Grossmann-Royer operators which originated from the problem of the
physical interpretation of the Wigner distribution, see \cite{Grossmann, Royer}.
It tuned out that the Grossmann-Royer operators 
are closely related to the Heisenberg operators, well-known objects from quantum mechanics, see \cite{deGosson2011, deGosson2017}.

The  Grossmann-Royer operator $ R : L^2 (\mathbb{R}^n) \rightarrow L^2 (\mathbb{R}^{2n})  $ is given by
\begin{equation} \label{GROp}
R f (x,\omega) =  R  (f(t) )(x,\omega)  = e^{4\pi i\omega (t-x)} f(2x - t), \;\;\; f \in L^2 (\mathbb{R}^n),\; x,\omega \in \mathbb{R}^n.
\end{equation}

We define the Grossmann-Royer transform $R_g f$ 
as the weak sense interpretation of the Grossmann-Royer operator.
The Grossmann-Royer transform is essentially the cross-Wigner distribution  $ W(f,g)$
(see Definition \ref{transformacije}, Lemma  \ref{GRandRelatives}
and \cite[Definition 12]{deGosson2017}), and is also closely related to  $V_g f$ and $A(f,g),$ the short-time Fourier transform and the cross-ambiguity transform respectively, see  Definition \ref{transformacije}.

\begin{Definition} \label{transformacije}
Let there be given $f, g \in L^2 (\mathbb{R}^n).$ The
Grossmann-Royer transform of $f$ and $g$ is given by
\begin{equation} \label{GRT}
R_g f (x,\omega) = R (f,g) (x,\omega)  = \langle R f, g \rangle   = \int e^{4\pi i  \omega(t-x)} f(2x- t) \overline{ g(t)} dt, \;\;\; x,\omega \in \mathbb{R}^n.
\end{equation}
The short-time Fourier transform of
$ f  $ with respect to the window $g$  is given by
\begin{equation} \label{GT}
V_g f (x,\omega) = \int e^{-2\pi i t \omega} f(t)  \overline{ g(t-x)}dt, \;\;\; x,\omega \in \mathbb{R}^n.
\end{equation}
The cross--Wigner distribution of  $f$ and $g$ is
\begin{equation} \label{WD}
W(f,g) (x,\omega)  = \int e^{-2\pi i \omega t} f(x+ \frac{t}{2})
\overline{ g(x- \frac{t}{2})} dt, \;\;\; x,\omega \in \mathbb{R}^n,
\end{equation}
and the cross--ambiguity function of  $f$ and $g$ is
\begin{equation} \label{FWT}
A (f,g) (x, \omega) = \int e^{-2\pi i \omega t}
f(t+ \frac{x}{2}) \overline{ g(t- \frac{x}{2})} dt, \;\;\; x,\omega \in \mathbb{R}^n.
\end{equation}
\end{Definition}

The transforms in Definition \ref{transformacije} are called 
time-frequency representations and we set
$$ 
TFR = \{ R_g f, V_g f, W (f,g), A(f,g)\}.
$$
 
Note that $ R  (f(t) )(x,\omega)  = e^{-4\pi i\omega x}  M_{2\omega}(T_{2x} f)^{\check{}} (t) $ so that
\begin{equation} \label{GRT-trans-mod}
R_g f (x,\omega)  =  e^{-4\pi i\omega x} \langle M_{2\omega} (T_{2x} f)^{\check{}} , g \rangle,
\end{equation}
where translation and modulation operators are respectively given by
$$
T_x f(\cdot) = f(\cdot - x) \;\;\; \mbox{ and } \;\;\;
 M_x f(\cdot) = e^{2\pi i x \cdot} f(\cdot), \;\;\; x \in \mathbb{R}^n,
$$
and $ \check{f} $ denotes the reflection $\check{f} (x) = f (-x)$.

By using appropriate change of variables the following relations between the time-frequency representations can be obtained (cf. \cite{Teo2018}).

\begin{Lemma} \label{GRandRelatives}
Let  $f, g \in L^2 (\mathbb{R}^n).$ Then we have:
\begin{eqnarray*}
W(f,g) (x,\omega)  & = & 2^n R_g f (x,\omega), \\
V_g f (x,\omega) & = &  e^{-\pi i x \omega} R_{\check{g}} f \left(\frac{x}{2},\frac{\omega}{2}\right),
\\
A(f,g) (x,\omega) & = & R_{\check{g}} f \left(\frac{x}{2},\frac{\omega}{2}\right), \quad
x,\omega \in \mathbb{R}^n.
\end{eqnarray*}
\end{Lemma}


The Grossmann-Royer operator is self-adjoint, uniformly continuous on
$\mathbb{R}^{2n}$, and the following properties hold:
\begin{enumerate}
\item $ \| R_g f \|_\infty \leq \| f\| \|g\|$;
\item $  R_g f = \overline{R_f g}$;
\item $  R_{ \hat g} \hat f ( x, \omega) =  R_{ g} f (-\omega, x),$
$ x,\omega \in  \mathbb{R}^n$;
\item For $ f_1, f_2, g_1, g_2  \in L^2 (\mathbb{R}^d),$ the Moyal identity holds:
$$ \langle R_{g_1} f_1, R_{ g_2} f_2 \rangle = \langle f_1, f_2 \rangle \overline{\langle g_1, g_2 \rangle}. $$
\end{enumerate}

The Moyal identity formula implies the inversion formula:
\begin{equation} \label{eq:inversion}
f = \frac{1}{\langle g_2, g_1 \rangle} \iint  R_{g_1} f (x, \omega)
R g_2 (x,\omega) \, dx d\omega.
\end{equation}
We refer to \cite{Gro} for details (explained in terms of the short-time Fourier transform).

Clearly, a smooth function $ F $ belongs to
$ \mathcal S_{(\mathcal{M}_\sigma)}^{(\mathcal{M}_\sigma)} 
(\mathbb R^{2n}) $ if
\begin{equation} \label{eq:doubledimension}
 (\forall \tau>0) (\exists C>0)
 \quad
   \sup_{x, \omega}
   |x^{\alpha_1} \omega^{\alpha_2} 
   \partial_x ^{\beta_1}   \partial_\omega ^{\beta_2} 
   F(x, \omega)|
   \leq C
   M^{\tau,\sigma}_{|\alpha_1 + \alpha_2|}  
   M^{\tau,\sigma}_{|\beta_1 + \beta_2|},
\end{equation}
for all $
\alpha_1, \alpha_2,\beta_1, \beta_2 \in\mathbb N_0^n$,
and similarly for $ \mathcal S_{\{\mathcal{M}_\sigma \}}^{\{\mathcal{M}_\sigma\}} 
(\mathbb R^{2n}) $.

 If $ F \in  \mathcal S_{[\mathcal{M}_\sigma]}^{[\mathcal{M}_\sigma]}  (\mathbb R^{2n}) $, then by calculations in Section \ref{Sec3}  it follows that $\mathcal F_1 F, \mathcal F_2 F  \in  \mathcal S_{[\mathcal{M}_\sigma]}^{[\mathcal{M}_\sigma]}  (\mathbb R^{2n}) $, where  $\mathcal F_1 $ and  
  $\mathcal F_2 $ denote partial Fourier transforms with respect to $x$ and $\omega ,$ respectively.

The following theorem, in the context of classical Gelfand-Shilov spaces, is considered folklore (see, for example,   \cite{GZ, T2, Teo2018, Toft-2012}).

\begin{Theorem} \label{nec-suf-cond}
Let $\sigma>1$, and let the weight matrix $\mathcal{M}_\sigma$ be given by 
\eqref{eq:the-weight-matrix}, 
$ M_{p}^{\tau,\sigma}=p^{\tau p^\sigma}$, $ p\in \mathbb{N}$, $ M_0^{\tau,\sigma}=1$.
If $f,g \in    \mathcal S_{[\mathcal{M}_\sigma ]}^{[\mathcal{M}_\sigma]} (\mathbb R^{n}) $, 
and  $ T \in TFA$, then $ T   \in   \mathcal S_{[\mathcal{M}_\sigma ]}^{[\mathcal{M}_\sigma]} (\mathbb R^{2n})$.

Conversely, if  $ T \in TFA$ and $ T   \in   \mathcal S_{[\mathcal{M}_\sigma ]}^{[\mathcal{M}_\sigma]} (\mathbb R^{2n})$
then $f,g \in     \mathcal S_{[\mathcal{M}_\sigma ]}^{[\mathcal{M}_\sigma]} (\mathbb R^{n}).$
\end{Theorem}

\begin{proof}
Since $\mathcal S_{[\mathcal{M}_\sigma ]}^{[\mathcal{M}_\sigma]} (\mathbb R^{n})$
are closed under reflections, dilations and modulations, then by Lemma \ref{GRandRelatives}
it suffices to give the proof for any $T \in TFA$.
We will give the proof for the Grossmann-Royer transform, and the same conclusion holds for other time-frequency representations in question. 

\par

Put $ \Phi (x,t) := f(2x -t) g (t)$, $x,t \in \mathbb R^{n}$.
Then 
$$
R_g f (x,\omega) = e^{-4\pi i \omega x}  \int e^{2\pi i \omega (2t) } \Phi (x,t)\, dt. 
$$
The first part of the claim will be proved 
if we show that
$ \Phi \in \mathcal S_{[\mathcal{M}_\sigma ]}^{[\mathcal{M}_\sigma]} (\mathbb R^{2n}) $,
since $  \mathcal S_{[\mathcal{M}_\sigma ]}^{[\mathcal{M}_\sigma]} (\mathbb R^{2n})$ is closed under dilations and modulations,
and since then $\mathcal F_2 \Phi \in 
  \mathcal S_{[\mathcal{M}_\sigma ]}^{[\mathcal{M}_\sigma]} (\mathbb R^{2n})$.
  
As before, we give the proof for the Beurling case, and the 
Roumieu case follows by similar arguments.

By Theorem \ref{Thm:opisFurije} it is enough to show
\begin{equation} \label{varphi po x}
\sup_{x,t \in \mathbb{R}^n}  
| x^{\alpha}  t^\beta  \Phi (x,t) | \leq C 
M^{\tau, \sigma} _{|\alpha|} M^{\tau, \sigma} _{|\beta|},\quad \alpha,\beta\in\mathbb N_0^n,
\end{equation}
and
\begin{equation} \label{varphi po ksi}
\sup_{x,t \in \mathbb{R}^n}  
|  \partial^\alpha _x \partial^\beta _t  \Phi (x,t) |
\leq C M^{\tau, \sigma} _{|\alpha|} M^{\tau, \sigma} _{|\beta|},\quad \alpha,\beta\in\mathbb N_0^n,
\end{equation}
for any given $\tau > 0$, and some constant $C= C_\tau >0$.

The decay of  $ \Phi (x,t) := f(2x -t) g (t)$ can be estimated as follows.
\begin{align}
\sup_{x,t \in \mathbb{R}^d}  | x^{\alpha}  t^\beta  f(2x-t) g (t) |
& \leq  
2^{-|\alpha|}\sup_{y,t \in \mathbb{R}^n}  |y+t|^{|\alpha|}  |t|^{|\beta|} 
|f(y) | |  g (t) | 
\nonumber \\
&\leq C 
2^{-|\alpha|}\sup_{y,t \in \mathbb{R}^n }
\sum_{\gamma \leq \alpha} \binom{\alpha}{\gamma}
|y|^{|\alpha|}  |t|^{|\alpha - \gamma| + |\beta|} 
|f(y) | |  g (t) | 
\nonumber \\
&\leq C  
\sup_{y \in \mathbb{R}^n} |y|^{|\alpha|}  |f(y) | \cdot
\sup_{t \in \mathbb{R}^n} |t|^{|\alpha| + |\beta|} 
 |  g (t) | 
\nonumber \\
& \leq
C  M^{\tau/2,\sigma}_{|\alpha |}  
M^{\tau/2,\sigma}_{|\alpha | +|\beta|} 
\nonumber \\
& \leq
C  M^{\tau,\sigma}_{|\alpha | + |\beta|}, 
\label{eq:est-mult}
\end{align}
see also Remark \ref{rem:alfa-beta}.


For the proof of  \eqref{varphi po ksi}, we use the Leibniz formula
and properties of the sequence $(M_{p}^{\tau,\sigma})_{p\in\mathbb N_0}$:

\begin{eqnarray*}
 \sup_{x,t \in \mathbb{R}^n} |\partial^\alpha _x \partial^\beta _t   f(2x-t) g (t)|
& \leq &
 2^{|\alpha|} \sum_{\gamma \leq \beta}
  \binom{\beta}{\gamma} 
\sup_{x,t \in \mathbb{R}^n} |\partial^{\alpha} _x \partial^{\gamma} _t  f (2x - t)
 \partial^{\beta - \gamma} _t g  (t)|
\\
& \leq & 2^{|\alpha| }  \sum_{\gamma \leq \beta}
  \binom{\beta}{\gamma} 
\sup_{x,t \in \mathbb{R}^n}
| \partial^{\alpha} _x \partial^{\gamma} _t  f (2x - t)|
\sup_{t \in \mathbb{R}^n} |\partial^{\beta - \gamma} _t g  (t)|
\\
& \leq & 2^{|\alpha| +|\beta|} 
C M^{\tau/4,\sigma}_{|\alpha |  +|\beta|}  
M^{\tau/4,\sigma}_{|\beta|} 
\leq  C_1 M^{\tau,\sigma}_{|\alpha |  +|\beta|}, 
\end{eqnarray*}
for some constants $C, C_1 > 0 $.

Thus $ R_g f \in \mathcal S_{(\mathcal{M}_\sigma )}^{(\mathcal{M}_\sigma)} (\mathbb R^{2n}) $ if  $f,g \in    \mathcal S_{(\mathcal{M}_\sigma )}^{(\mathcal{M}_\sigma)} (\mathbb R^{n}) .$


Now, assume that 
\begin{equation} \label{eq:RinGSspace}
R_g f \in  
\mathcal S_{(\mathcal{M}_\sigma )}^{(\mathcal{M}_\sigma)} (\mathbb R^{2n}) 
\end{equation}
for a given $g \in L^2 (\mathbb R^{2n}) \setminus \{ 0\}$.

By the inversion formula \eqref{eq:inversion}
we have
\begin{equation} \label{eq:term-for-estimate}
t^\alpha \partial^\beta _t f (t) = \frac{1}{\langle h, g \rangle} 
\iint R_{g} f (x, \omega)
t^\alpha \partial^\beta _t (R (h (t)) (x, \omega)\, dx d\omega,
\end{equation}
where we may choose $h \in \mathcal S_{(\mathcal{M}_\sigma )}^{(\mathcal{M}_\sigma)} (\mathbb R^{n})$ such that 
$\langle h, g \rangle = 1$.

Let there be given $\tau >0$.

From \eqref{eq:RinGSspace} we get
\begin{eqnarray}
 | t^\alpha \partial^\beta _t (R (h (t)) (x, \omega) |
&  = &
| t^\alpha \partial^\beta _t  ( e^{4\pi i \omega (t-x)} h(2x-t))|
\nonumber \\
&  \leq &
C \sum_{\gamma \leq \beta} \binom{\beta}{\gamma}
|\omega^{\beta - \gamma} |
 | t^{\alpha}  \partial^\gamma _t h(2x-t))| 
 \nonumber \\
&  \leq &
C M^{\tau/2, \sigma} _{|\alpha|} M^{\tau/2, \sigma} _{|\beta|} 
\sum_{\gamma \leq \beta} \binom{\beta}{\gamma}
|\omega^{\beta - \gamma} |, 
\nonumber
\end{eqnarray}
wherefrom
\begin{equation} \label{eq:term-for-estimate-2}
|t^\alpha \partial^\beta _t f (t) |
\leq C
M^{\tau/2, \sigma} _{|\alpha|} M^{\tau/2, \sigma} _{|\beta|} 
\sum_{\gamma \leq \beta} \binom{\beta}{\gamma}
\iint |\omega^{\beta - \gamma}  R_{g} f (x, \omega)|
dx d\omega.
\end{equation}
Next we estimate the double integral in \eqref{eq:term-for-estimate-2}. We use  similar argument as in the proof of Theorem \ref{thm:L2-and-Linfty} (cf. \eqref{eq:Linfty-implies-L2}):
\begin{align}
\iint |\omega^{\beta - \gamma}  R_{g} f (x, \omega) & |
 dx d\omega \nonumber \\
  & \leq 
C_1 \sup_{x,\omega \in \mathbb R^{n}} (1+|x|)^2
(1+|\omega|)^{2+|\beta - \gamma|} |R_{g} f (x, \omega)|
\nonumber \\
  & \leq 
C_2    M^{\tau/4, \sigma} _{2} M^{\tau/4, \sigma} _{2+|\beta -\gamma|}  \leq C_3 M^{\tau/4, \sigma} _{2+|\beta|} 
\nonumber \\
  & \leq 
C_4 ^{(2^{\sigma - 1} +1) |\beta|^\sigma} 
 M^{\tau/4, \sigma} _{|\beta|} 
 \nonumber
\end{align}
for some constants $ C_j >0,$ $ j =1,2,3,4$, wherefrom
\begin{equation}
\sum_{\gamma \leq \beta} \binom{\beta}{\gamma}
\iint |\omega^{\beta - \gamma}  R_{g} f (x, \omega)|
dx d\omega
\leq 
C_5 ^{(2^{\sigma - 1} +1) |\beta|^\sigma} 
M^{\tau/4, \sigma} _{|\beta|} 
\leq
 C M^{\tau/2, \sigma} _{|\beta|},
  \label{eq:Restimate-2}
\end{equation}
for some constants $C, C_5 > 0$.
By \eqref{eq:term-for-estimate-2} and 
\eqref{eq:Restimate-2} we get
$$
\sup_{t \in \mathbb R^{n}}
|t^\alpha \partial^\beta _t f (t) | \leq
C 
M^{\tau/2, \sigma} _{|\alpha|} M^{\tau/2, \sigma} _{|\beta|}  M^{\tau/2, \sigma} _{|\beta|}
\leq
C 
M^{\tau, \sigma} _{|\alpha|} M^{\tau, \sigma} _{|\beta|},  
$$
i.e. $ f \in \mathcal S_{(\mathcal{M}_\sigma )}^{(\mathcal{M}_\sigma)} (\mathbb R^{n}) $.

Finally, $ g \in \mathcal S_{(\mathcal{M}_\sigma )}^{(\mathcal{M}_\sigma)} (\mathbb R^{n}) $
follows from 
$  R_g f = \overline{R_f g}$.
\end{proof}

\par

\begin{Remark}
An alternative proof of Theorem \ref{nec-suf-cond} can be given using arguments based on the kernel theorem for Gelfand-Shilov spaces with extended Gevrey regularity. This is beyond the scope of the current contribution and will be studied elsewhere. For details on kernel theorems and their connection to nuclearity, we refer to \cite{DNV2021}.
\end{Remark}






\section{Acknowledgments}


This research was supported by the Science Fund of the Republic of
Serbia, $\#$GRANT No. 2727, {\it Global and local analysis of operators and
distributions} - GOALS.  The authors were also supported by the Ministry of Science, Technological Development, and Innovation of the Republic of Serbia— the first and third authors by Grants No. 451-03-137/2025-03/200125 and 451-03-136/2025-03/200125, and the second author by Grant No. 451-03-136/2025-03/200156.

\end{document}